\documentclass[reqno,10pt]{amsart}

\usepackage{amsmath, amsthm}
\usepackage{hyperref}
\usepackage{pdfsync}
\usepackage{parskip}
\numberwithin{equation}{section}

\newtheorem{theorem}{Theorem}[section]
\newtheorem{prop}[theorem]{Proposition}
\newtheorem{definition}[theorem]{Definition}

\def\R{\mathbb{R}}
\def\M{\sf{M}}
\def\B{\mathbb{B}(0,r)}
\def\U{\sf{O}}
\def\Uk{\sf{O}_{\kappa}}
\def\Q{\mathbb{B}^{m}}
\def\Qt{\mathbb{B}^{2}}
\def\pk{\pi_\kappa}
\def\F{\mathfrak{F}}
\def\N{\mathbb{N}}
\def\ttm{{\Theta}^{\ast}_{\mu}}
\def\ttl{{\Theta}^{\ast}_{\lambda,\mu}}
\def\Xt{\zeta_{1}}
\def\rh{{\varrho}^{\ast}_{\lambda}}
\def\tu{{T}_{\mu}}
\def\ez{\mathbb{E}_{0}}
\def\ef{\mathbb{E}_{1}}
\def\efa{\mathbb{E}_{1}^{a}}

\begin{document}

\title{Real Analytic Solutions to the Willmore Flow}
\author{Yuanzhen Shao}
\address{Department of Mathematics,
         Vanderbilt University, 
         Nashville, TN 37240, USA}
\email{yuanzhen.shao@vanderbilt.edu}

\date{December 29, 2012}

\let\thefootnote\relax
\let\thefootnote\relax

\thispagestyle{empty}

\begin{abstract}
In this paper, a regularity result for the Willmore flow is presented. It is established by means of a truncated translation technique in conjunction with the Implicit Function Theorem.
\end{abstract}
\maketitle

\section{\bf Introduction}
\medskip
The Willmore flow consists in looking for an oriented, closed, compact moving hypersurface $\Gamma(t)$ immersed in ${\R}^{3}$ evolving subject to the law 
\begin{equation}
\label{original eq 1.1}
\begin{cases}
V(t)=-\Delta_{\Gamma(t)}H_{\Gamma(t)}-2H_{\Gamma(t)}(H^2_{\Gamma(t)}-K_{\Gamma(t)}),\\
\Gamma(0)=\Gamma_0.
\end{cases} 
\end{equation}
Here $V(t)$ denotes the velocity in the normal direction of $\Gamma$ at time $t$ and $\Delta_{\Gamma(t)}$ stands for the Laplace-Beltrami operator, while $H_{\Gamma(t)}$ is the normalized mean curvature of $\Gamma(t)$. Finally, $K_{\Gamma(t)}$ denotes the Gaussian curvature.
\smallskip\\
The equilibria of \eqref{original eq 1.1} appear as the critical points of the Willmore functional, or sometimes called the Willmore energy. For a smooth immersion $f:{\Gamma}\rightarrow{\R}^{3}$ of a closed oriented two-dimensional manifold $\Gamma$, the Willmore functional is defined as:
\begin{equation}
\label{Willmore functional}
W(f)=\int\limits_{f(\Gamma)} H_{f(\Gamma)}^{2}\, d\sigma{,}
\end{equation}
where $d\sigma$ is the area element on $f(\Gamma)$ with respect to the Euclidean metric in ${\R}^3$. The critical surfaces of this functional, called the Willmore surfaces, satisfy the equation:
\begin{equation}
\label{Euler-Lagrange equation}
\Delta_{f(\Gamma)}H_{f(\Gamma)}+2H_{f(\Gamma)}^{3}-2H_{f(\Gamma)}K_{f(\Gamma)}=0{.}
\end{equation}
The reader may consult \cite[Section~7.4]{TJWR} for a brief historical account and a proof of this variational formula. The proof therein is derived by computing the critical points of all normal variations of the hypersurface ${f(\Gamma)}$. 
\smallskip\\
A generalization of the Willmore functional \eqref{Willmore functional} in higher dimensions is studied by B.-Y.~Chen \cite{BYCV}. He extends \eqref{Willmore functional} for smooth immersions $f:{\Gamma}\rightarrow{\R}^{m+1}$ of the $m$-dimensional closed oriented manifold $\Gamma$ into ${\R}^{m+1}$:
\begin{align*}
W(f)=\int\limits_{f(\Gamma)} H_{f(\Gamma)}^{m}\, d\sigma
\end{align*}
with $d\sigma$ standing for the volume element with respect to the Euclidean metric in ${\R}^{m+1}$. The critical points of this functional are now of the form:
\begin{align*}
\Delta_{f(\Gamma)}H_{f(\Gamma)}^{m-1}+m(m-1)H_{f(\Gamma)}^{m+1}-H_{f(\Gamma)}^{m-1}R_{f(\Gamma)}=0{.}
\end{align*}
Here $R_{f(\Gamma)}$ denotes the scalar curvature. We may observe that $R_{f(\Gamma)}=2K_{f(\Gamma)}$ when $m=2$, so this Euler-Lagrange equation agrees with \eqref{Euler-Lagrange equation} in the two-dimensional case. However, this generalization has the drawback that the corresponding Willmore functional is no longer conformally invariant except when $m=2$.
\smallskip\\
The Willmore problem has been studied by many authors, among them T.J.~Willmore, W.~Blaschke, B.-Y.~Chen, J.L.~Weiner, P.~Li, S.-T.~Yau, R.~Bryant, R.~Kusner, L.~Simon, U.F. Mayer, G.~Simonett, M.~Bauer, E.~Kuwert, R.~Sch\"atzle, U.~Pinkall, I.~Sterling, M.U.~Schmidt, C.M.~Fernando, and N.~Andr\'e. See for example \cite{MBEK,RBDW,BYCV,FAWC,RKCW,EKSE,KSGF,EKRS,LYCW,MSSI,MSWF,UPHT,PSWS,MSWC,LSMW,GSWF,JWPC,TJWR}. It is well-known that the Willmore functional is bounded below by $4\pi$ with equality only for the round sphere. Then the famous Willmore conjecture due to T.J.~Willmore asserts that for any immersed $2$-dimensional torus into $\R^3$ we have $W(f)\geq{2\pi^2}$, and it suggests that the $2$-dimensional Clifford torus achieves the minimum of the Willmore functional amongst all immersed tori in $\R^3$. In 1982, P.~Li and S.-T.~Yau \cite{LYCW} show that any immersion with $W(f)<8\pi$ must in fact be an embedding. In other words, it will suffice to estimate $W(f)$ for embeddings. A classification of all Willmore immersions $f:\mathbb{S}^2\rightarrow\R^3$ is obtained by R.L.~Bryant \cite{RBDW}. The possible values of 
\begin{align*}
W(f)=\int\limits_{f(\mathbb{S}^2)} H_{f(\mathbb{S}^2)}^{2}\, d\sigma
\end{align*}
are $4n\pi$ with $n=1$, or $n\geq{4}$ and $n$ even, or $n\geq{9}$ and $n$ odd. Existence and regularity for embedded tori in the Willmore conjecture has been proven by L.~Simon \cite{LSMW}, and later this result is generalized by M.~Bauer, E.~Kuwert \cite{MBEK} for an extension of the conjecture by R.~Kusner \cite{RKCW} to higher genus cases. An existence, uniqueness and regularity result on the Willmore flow is presented by G.~Simonett \cite{GSWF}. It is proven therein that the Willmore flow admits a unique smooth solution. Moreover, this solution exists globally when it is initially close enough to spheres in the $C^{2+\alpha}$-topology and is exponentially attracted by spheres. In \cite{MSSI}, U.F Mayer and G. Simonett  prove that the Willmore flow can drive embedded surfaces to a self-intersection in a finite time interval. Moreover, numerical simulations in \cite{MSWF} indicate that the Willmore flow can develop true singularities (topological changes) in finite time. E.~Kuwert and R.~Sch\"atzle \cite{EKSE} show that the smooth solutions are global as long as the initial Willmore energy is sufficiently small. Later, the same authors improve this result in \cite{EKRS} by finding an explicit optimal bound for the restriction on the initial energy, that is, if the smooth immersion $f_0:\Gamma\rightarrow\R^3$ satisfies $W(f_0)\leq{8\pi}$, then the solution with initial data $f_0$ exists smoothly for all time and converges to a round sphere. Recently, in a breakthrough paper, C.M.~Fernando and N.~Andr\'e \cite{FAWC} prove the Willmore conjecture for surfaces of arbitrary genus $g\geq 1$, i.e., $W(f)\geq 2\pi^2$ for all embedded $\Gamma$ with genus $g\geq 1$, and the equality holds iff $\Gamma$ is conformal to the Clifford torus.
\medskip\\ 

{\textbf {Assumptions:}}
Throughout this paper, we always assume that $({\M},g)$ is a compact, closed, immersed, oriented, real analytic hypersurface in ${\R}^{3}$ endowed with the Euclidean metric $g$ with the exception of Section~3, wherein we remove the restriction on the dimension of ${\M}$. The notation $(\cdot|\cdot)$ always stands for the standard inner product in ${\R}^{3}$. We may find for ${\M}$ a normalized atlas $({\Uk},\varphi_{\kappa})_{{\kappa}\in \Lambda}$, where an atlas is said to be normalized if $\varphi_{\kappa}({\Uk})= {\Qt}$ for all $\kappa\in\Lambda$. Here ${\Qt}$ is the open unit ball centered at the origin in ${\R}^2$. Put $\psi_{\kappa}=\varphi_{\kappa}^{-1}$. 
\smallskip\\
A family $({\pk})_{{\kappa}\in\Lambda}$ is called a localization system subordinate to $({\Uk},\varphi_{\kappa})_{{\kappa}\in \Lambda}$ if:
\begin{itemize}
\item[(L1)] ${\pk}\in\mathcal{D}({\Uk},[0,1])$ and $(\pi_{\kappa}^{2})_{\kappa\in{\Lambda}}$ is a partition of unity subordinate to the open cover $({\Uk})_{{\kappa}\in\Lambda}$. 
\item[(L2)] Any ${\pk}$ and $\pi_{\eta}$ satisfying ${\rm{supp}}({\pk})\cap {\rm{supp}}({\pi_\eta})\neq \emptyset$ have their supports located within the same local chart.
\end{itemize}
For any manifold satisfying the above assumptions, there exists a localization system. See \cite[Lemma~3.2]{FSM} for a proof. 
\medskip\\

\textbf{Notations:} 
Throughout this paper, ${\N}_0$ stands for the set of all natural numbers including $0$. For any time interval $I$, $\dot{I}$ always denotes the interior of $I$.
\smallskip\\
Fix $0<\alpha<1$. Let $\gamma\in(0,1]$ and $E_0:=h^{\alpha}({\M})$, $E_1:=h^{4+\alpha}({\M})$. For notational brevity, we simply write $\F(\mathcal{O},\R)$ and $\F({\M},\R)$ as $\F(\mathcal{O})$ and $\F({\M})$, where $\mathcal{O}$ is an open subset of $\R^2$ and $\F$ stands for any of the function spaces in this paper.
\smallskip\\
In the sequel, we always denote $(E_0,E_1)_{\gamma}$ by $E_{\gamma}$, where $(\cdot,\cdot)_{\gamma}$ is the continuous interpolation method. See \cite[Definition~1.2.2]{LSOR} for a definition. Note that the continuous interpolation method $(\cdot,\cdot)_{\theta,\infty}^0$ coincides with $(\cdot,\cdot)_{\theta}$ in the suggested reference. In particular, we set $(E_0,E_1)_1:=E_1$. 
\smallskip\\
For some fixed interval $I=[0,T]$ and some Banach space $E$, we define
\begin{align*}
&BU\!C_{1-\gamma}(I,E):=\{u\in{C(\dot{I},E)};[t\mapsto{t^{1-\gamma}}u]\in{BU\!C(\dot{I},E)},\lim_{t\to{0^+}}{t^{1-\gamma}}\|u\|=0\},\\
& \|u\|_{C_{1-\gamma}}:=\sup_{t\in{\dot{I}}}{t^{1-\gamma}}\|u(t)\|_{E},
\end{align*}
and
\begin{center}
$BU\!C_{1-\gamma}^1(I,E):=\{u\in{C^1(\dot{I},E)}: u,\dot{u}\in{BU\!C_{1-\gamma}(I,E)}\}$.
\end{center}
In particular, we put
\begin{center}
$BU\!C_0(I,E):=BU\!C(I,E)$\hspace*{1em} and \hspace*{1em} $BU\!C^1_0(I,E):=BU\!C^1(I,E)$.
\end{center}
In addition, if $I=[0,T)$ is a half open interval, then
\begin{align*}
&C_{1-\gamma}(I,E):=\{v\in{C(\dot{I},E)}:v\in{BU\!C_{1-\gamma}([0,t],E)},\hspace{.5em} t<T\},\\
&C^1_{1-\gamma}(I,E):=\{v\in{C^1(\dot{I},E)}:v,\dot{v}\in{C_{1-\gamma}(I,E)}\}.
\end{align*}
We equip these two spaces with the natural Fr\'echet topology induced by the topology of $BU\!C_{1-\gamma}([0,t],E)$ and $BU\!C_{1-\gamma}^1([0,t],E)$, respectively. 
\smallskip\\
Last but not least, we set
\begin{center}
${\ez}(I):=C(I,E_0)$\hspace*{1em} and \hspace*{1em} ${\ef}(I):=C(I,E_1)\cap{C}^1(I,E_0)$.
\end{center}
\medskip
\goodbreak

It will be shown in this paper that the Willmore flow \eqref{original eq 1.1} admits a real analytic solution jointly in time and space. Our motivation for a real analytic solution is mainly stimulated by the following facts: a compact closed real analytic manifold cannot have a "flat part", and real analyticity in time implies that the hyersurface should move permanently in the interval of existence.
\begin{theorem}
\label{main theorem}
Let $0<\alpha<1$. Suppose that $\Gamma_0$ is a compact closed immersed oriented hypersurface in ${\R}^{3}$ belonging to the class $h^{2+\alpha}$. Then the Willmore flow \eqref{original eq 1.1} has a unique local solution $\Gamma=\{\Gamma(t):t\in[0,T)\}$ for some $T>0$. Moreover,
\begin{align*}
\mathcal{M}:=\bigcup_{t\in(0,T)}(\{t\}\times\Gamma(t))
\end{align*}
is a real analytic submanifold in ${\R}^{4}$. In particular, each manifold $\Gamma(t)$ is real analytic for $t\in(0,T)$. 
\end{theorem}
For any open subset $\mathcal{O}\subset{\R}^{2}$, the little H\"older space $h^{s}(\mathcal{O})$ of order $s>0$ with $s\notin{\N}$ is the closure of $BU\!C^{\infty}(\mathcal{O})$ in $BU\!C^{s}(\mathcal{O})$. Here $BU\!C^{s}(\mathcal{O})$ is the Banach space of all bounded and uniformly H\"older continuous functions. The little H\"older space $h^{s}(\M)$ on $\M$ is defined in terms of a smooth atlas, that is, a function $u$ belongs to $h^{s}({\M})$ iff $\psi_{\kappa}^{\ast}{\pk}u\in h^{s}({\R}^{2})$, for each $\kappa\in\Lambda$.
\bigskip
\vspace{1em}

\section{\bf Parameterization over a Reference Manifold}
\medskip
In equation \eqref{original eq 1.1}, if we fix an initial hypersurface $\Gamma_0$ belonging to the class $h^{2+\alpha}$, then by the discussion in \cite[Section~4]{MCH} we can find a real analytic compact closed embedded oriented hypersurface ${\M}$, a function $\rho_0 \in h^{2+\alpha}(\M) $ and a parameterization 
\begin{center}
$\Psi_{\rho_0}:{\M}\rightarrow{\R}^{3}$,\hspace{.5em} $\Psi_{\rho_0}(p):=p+\rho_0(p){\nu}_{\M}(p)$
\end{center}
such that $\Gamma_0={\rm{im}}(\Psi_{\rho_0})$. Here ${\nu}_{\M}(p)$ denotes the unit normal with respect to a chosen orientation of $\M$ at $p$, and $\rho_0:{\M}\rightarrow (-a,a)$ is a real-valued function on ${\M}$, where $a$ is a sufficiently small positive number depending on the inner and outer ball condition of ${\M}$. The reader may consult \cite[Section~4.1]{MCH} for the precise bound of $a$. Thus $\Gamma_0$ lies in the $a$-tubular neighborhood of ${\M}$. In fact, it will suffice to assume $\Gamma_0$ to be a $C^2$-manifold for the existence of such a parameterization and a real analytic reference manifold. See \cite[Section~4]{MCH} for a detailed proof.
\smallskip\\
Analogously, if $\Gamma(t)$ is $C^1$-close enough to ${\M}$, then we can find a function $\rho:[0,T)\times{\M}\rightarrow(-a,a)$ for some $T>0$ and a parameterization
\begin{center}
$\Psi_{\rho}:{[0,T)}\times{\M}\rightarrow{\R}^{3}$,\hspace{.5em} $\Psi_{\rho}(t,p):=p+\rho(t,p){\nu}_{\M}(p)$
\end{center}
such that $\Gamma(t)={\rm{im}}(\Psi_{\rho}(t,\cdot))$ for every $t\in [0,T)$. It is worthwhile to mention that $\Psi_{\rho}$ admits an extension on $\R^{3}$, called Hanzawa transform, which was first introduced by E.I.~Hanzawa in \cite{HCSS}.
\smallskip\\
For any fixed $t$, I do not distinguish between $\rho(t,\cdot)$ and $\rho(t,\psi_{\kappa}(\cdot))$ in each local coordinate $({\Uk},\varphi_{\kappa})$ and abbreviate $\Psi_{\rho}(t,\cdot)$ to be $\Psi_{\rho}:=\Psi_{\rho}(t,\cdot)$. In addition, the hypersurface $\Gamma(t)$ will be simply written as $\Gamma_{\rho}$ as long as the choice of $t$ is of no importance in the context, or $\rho$ is independent of $t$.
\smallskip\\
We put 
\begin{align*}
\mho:=\{\rho\in{h^{2+\alpha}({\M})}: \|\rho\|_{\infty}^{\M}<a\}. 
\end{align*}
Here $\|\rho\|_{\infty}^{\M}:=\sup_{p\in{\M}}|\rho(p)|$. For any $\rho\in\mho$, ${\rm{im}}(\Psi_{\rho})$ constitutes a $h^{2+\alpha}$-hypersurface $\Gamma_{\rho}$. In this case, $\Psi_{\rho}$ defines a $h^{2+\alpha}$-diffeomorphism from ${\M}$ onto $\Gamma_{\rho}$. 
\smallskip\\
Here and in the following, it is understood that the Einstein summation convention is employed and all the summations run from $1$ to $2$ for all repeated indices. 

In \cite{MCH}, J.~Pr\"uss and G.~Simonett derive global expressions for many geometric objects of $\Gamma_{\rho}$ in terms of the function $\rho$. I will use some results therein to translate equation \eqref{original eq 1.1} into a differential equation in $\rho$. By \cite[formula~(23), (28)]{MCH}, we have the following explicit expressions for the components of the first fundamental form and the normal vector of $\Gamma_{\rho}$:
\begin{equation}
\label{gamma_ij}
g^{\Gamma}_{ij}=g_{ij}-2\rho{l}_{ij}+\rho^{2}l^{r}_{i}l_{jr}+\partial_{i}\rho\partial_{j}\rho,
\end{equation} 
and
\begin{equation}
\label{nu}
\nu_{\Gamma}=\beta(\rho)(\nu_{\M}-a(\rho)).
\end{equation}
In \eqref{gamma_ij}, the $l^i_j$'s are the components of the Weingarten tensor $L_{\M}$ of ${\M}$ with respect to $g$, i.e., $L_{\M}=l_{i}^{j}\tau^{i}\otimes\tau_j$, where $\{\tau_{i}=\partial_{i}\}$ forms a basis of $T_{p}{\M}$ at $p\in{\M}$ and $\{\tau^{i}\}$ is the dual basis to $\{\tau_{i}\}$, i.e., $(\tau^i|\tau_j)=\delta^i_j$. The extension of $L_{\M}$ into ${\R}^3$, by identifying it to be zero in the normal direction, is denoted by $L_{\M}^{\mathcal{E}}$, namely, $L_{\M}^{\mathcal{E}}=l_{i}^{j}\tau^{i}\otimes\tau_j+0\cdot\nu_{\M}\otimes\nu_{\M}$. It is a simple matter to check that
\begin{equation}
\label{standard basis}
\tau^{\Gamma}_{i}=(I-\rho{L_{\M}^{\mathcal{E}}})\tau_{i}+\nu_{\M}\partial_{i}\rho
\end{equation}
forms the standard basis of $T_{\Psi_{\rho}(p)}{\Gamma_{\rho}}$. In addition, the $l_{ij}$'s are the components of the second fundamental form $L^{\M}$ of the metric $g$. Finally, $g^{\Gamma}_{ij}=(\tau^{\Gamma}_{i}|\tau^{\Gamma}_{j})$ are the components of the first fundamental form of the Euclidean metric $g_{\Gamma}$ on $\Gamma_{\rho}$. We set $G^{\Gamma}(\rho)=(g^{\Gamma}_{ij})_{ij}$ and $G^{-1}_{\Gamma}(\rho)$ for its inverse.
\smallskip\\
In \eqref{nu}, the terms $a(\rho)$ and $\beta(\rho)$ read as
\begin{center}
$a(\rho)=(I-\rho{L_{\M}^{\mathcal{E}}})^{-1}\nabla_{\M}\rho$\hspace{1em} and\hspace{1em} $\beta(\rho)=[1+|a(\rho)|^2]^{-1/2}$. 
\end{center}
Here $\nabla_{\M}$ is the surface gradient on $\M$.
\smallskip\\
For sufficiently small $a>0$, $(I-\rho{L_{\M}^{\mathcal{E}}})$ is invertible. One can check that
\begin{align*}
I-\rho{L_{\M}^{\mathcal{E}}}=(\delta_{i}^{j}-\rho{l}_{i}^{j})\tau^{i}\otimes\tau_{j}+\nu_{\M}\otimes\nu_{\M}. 
\end{align*} 
Thus 
\begin{equation}
\label{Inverse formula}
(I-\rho{L_{\M}^{\mathcal{E}}})^{-1}=r_i^j(\rho)\tau^{i}\otimes\tau_{j}+\nu_{\M}\otimes\nu_{\M},
\end{equation}
where ${R}_{\rho}=(r_i^j(\rho))_{ij}=[(\delta_{i}^{j}-\rho{l}_{i}^{j})_{ij}]^{-1}$. By Cramer's rule, all the entries of ${R}_{\rho}$ possess the expression 
\begin{align*}
r_i^j(\rho)=\frac{P_{i}^{j}(\rho)}{Q_{i}^{j}(\rho)}
\end{align*}
in every local chart, where $P_{i}^{j}$ and $Q_{i}^{j}$ are polynomials in $\rho$ with real analytic coefficients and $Q_{i}^{j}\neq{0}$.
\smallskip\\
Substituting $(I-\rho{L_{\M}^{\mathcal{E}}})^{-1}$ by \eqref{Inverse formula}, we get
\begin{align*}
|a(\rho)|^2=( r^j_i(\rho) \partial_j\rho\tau^i| r^l_k(\rho) \partial_l\rho\tau^k)=g^{ik} r_i^j(\rho) r_k^l(\rho) \partial_{j}\rho \partial_{l}\rho.
\end{align*}
Then
\begin{align*}
\beta(\rho)=[1+|a(\rho)|^2]^{-1/2}=[1+g^{ik} r_i^j(\rho) r_k^l(\rho) \partial_{j}\rho \partial_{l}\rho]^{-1/2}. 
\end{align*}
Note that in every local chart 
\begin{align*}
\beta^2(\rho)=\frac{P^{\beta}(\rho)}{Q^{\beta}(\rho,\partial_{j}\rho)}, 
\end{align*}
where $P^{\beta}(\rho)$ is a polynomial in $\rho$ with real analytic coefficients and $Q^{\beta}(\rho,\partial_{j}\rho)\neq{0}$ is a polynomial in $\rho$ and its first order derivatives with real analytic coefficients. 
\smallskip\\
The normal velocity can be expressed as 
\begin{align*}
V(t)=(\partial_{t}\Psi_{\rho}|\nu_{\Gamma})=(\rho_{t}\nu_{\M}|\nu_{\Gamma})=\beta(\rho)\rho_{t}.
\end{align*}
Therefore, the first line of equation \eqref{original eq 1.1} is equivalent to
\begin{align*}
\rho_{t}=-\frac{1}{\beta(\rho)}[\Psi^{\ast}_{\rho}\Delta_{\Gamma_{\rho}}H_{\Gamma_{\rho}}+2\Psi^{\ast}_{\rho}H_{\Gamma_{\rho}}(H^2_{\Gamma_{\rho}}-K_{\Gamma_{\rho}})].
\end{align*}
\smallskip\\

Next we shall calculate the Gaussian curvature $K_{\Gamma_{\rho}}$ in terms of $\rho$. For simplicity, we write $K_{\rho}$ instead of $\Psi^{\ast}_{\rho}K_{\Gamma_{\rho}}$.
\smallskip\\
Because 
\begin{center}
$\partial_j \tau_i=\Gamma^{k}_{ij}\tau_k+l_{ij}\nu_{\M}$\hspace{1em} and \hspace{1em}$\partial_j \tau^{i}=-\Gamma^i_{jk}\tau^k+l^i_j\nu_{\M}$,
\end{center}
we may readily compute 
\begin{equation}
\label{derivative of WT}
\partial_{j}L_{\M}^{\mathcal{E}}
=\partial_{j}l^{k}_{i}\tau^{i}\otimes\tau_{k}-\Gamma^{i}_{jl}l_{i}^{k}\tau^{l}\otimes\tau_{k}
+\Gamma^{l}_{jk}l^{k}_{i}\tau^{i}\otimes\tau_{l}+l_{j}^{i}l^{k}_{i}\nu_{\M}\otimes\tau_{k}
+l_{jk}l^{k}_{i}\tau^{i}\otimes\nu_{\M}.
\end{equation} 

Denote by $L^{\Gamma}=(l_{ij}^{\Gamma})_{ij}$ the second fundamental form of $\Gamma_{\rho}$ with respect to $g_{\Gamma}$. Then by \eqref{nu} and \eqref{standard basis}, we can compute its components $l_{ij}^{\Gamma}$ as follows:
\begin{align*}
l_{ij}^{\Gamma}&=-(\tau^{\Gamma}_{i}|\partial_{j}\nu_{\Gamma})\\
&=-((I-\rho{L_{\M}^{\mathcal{E}}})\tau_{i}+\nu_{\M}\partial_{i}\rho|\beta({\partial_{j}\nu_{\M}-\partial_{j}a(\rho)}))-(\tau^{\Gamma}_{i}|\frac{\partial_{j}\beta}{\beta}\nu_{\Gamma})\\
&=\beta\{{l_{ij}}+\rho(L_{\M}^{\mathcal{E}}\tau_{i}|\partial_{j}\nu_{\M})+(\tau_{i}|\partial_{j}(\nabla_{\M}\rho))+((I-\rho{L_{\M}^{\mathcal{E}}})\tau_{i}|\partial_{j}[(I-\rho{L_{\M}^{\mathcal{E}}})^{-1}]\nabla_{\M}\rho)\\
&\hspace*{1em}+\partial_{i}\rho(\nu_{\M}|\partial_{j}[(I-\rho{L_{\M}^{\mathcal{E}}})^{-1}]\nabla_{\M}\rho)+\partial_{i}\rho(\nu_{\M}|(I-\rho{L_{\M}^{\mathcal{E}}})^{-1}[\partial_{j}(\nabla_{\M}\rho)])\}\\
&=\beta\{{l_{ij}}+\rho{l_{ik}(\tau^{k}|\partial_{j}\nu_{\M})}+(\tau_{i}|\partial_{j}(\nabla_{\M}\rho))+(\tau_{i}|\partial_{j}(\rho{L_{\M}^{\mathcal{E}}})(I-\rho{L_{\M}^{\mathcal{E}}})^{-1}\nabla_{\M}\rho)\\
&\hspace*{1em}+\partial_{i}\rho(\nu_{\M}|\partial_{j}(\rho{L_{\M}^{\mathcal{E}}})(I-\rho{L_{\M}^{\mathcal{E}}})^{-1}\nabla_{\M}\rho)+\partial_{i}\rho(\nu_{\M}|\partial_{j}(\nabla_{\M}\rho))\}\\
&=\beta[{l_{ij}}-{l_{ik}}{l^{k}_{j}}\rho+{\partial_{ij}\rho}-\Gamma^{k}_{ij}\partial_{k}\rho
+r_k^l(\rho)(\partial_{j}l_{i}^{k}+\Gamma^{k}_{jh}l_{i}^{h}-\Gamma^{h}_{ij}l_{h}^{k})\rho\partial_{l}\rho \\
&\hspace*{1em}+r_k^l(\rho)l^{k}_{i}\partial_{j}\rho\partial_{l}\rho +r_k^l(\rho)l^{h}_{j}l^{k}_{h}\rho\partial_{i}\rho\partial_{l}\rho +{l^{k}_{j}}\partial_{i}\rho\partial_{k}\rho].
\end{align*}
We have used \eqref{derivative of WT} and the following facts in the above computation:
\begin{itemize}
\item $(\nu_{\M}|\partial_j\nu_{\M})=0$.
\item $(\tau^{\Gamma}_{i}|\nu_{\Gamma})=0$.
\item $\partial_j \nu_{\M}=-l_{ij}\tau^i$.
\item $(I-\rho{L_{\M}^{\mathcal{E}}})^{-1}\nu_{\M}=\nu_{\M}$.
\item $\partial_j a(\rho)=(I-\rho{L_{\M}^{\mathcal{E}}})^{-1}\partial_j(\nabla_{\M}\rho)+\partial_j[(I-\rho{L_{\M}^{\mathcal{E}}})^{-1}]\nabla_{\M}\rho$.
\item $\partial_j[(I-\rho{L_{\M}^{\mathcal{E}}})^{-1}]=(I-\rho{L_{\M}^{\mathcal{E}}})^{-1} \partial_j(\rho{L_{\M}^{\mathcal{E}}}) (I-\rho{L_{\M}^{\mathcal{E}}})^{-1}$.
\end{itemize}

Therefore, $ {\det}(L^{\Gamma})$ can be expressed in every local chart as
\begin{align*}
{\det}(L^{\Gamma})=\beta^{2}(\rho)\frac{P^{\Gamma}(\rho,\partial_{j}\rho,\partial_{ij}\rho)}{Q^{\Gamma}(\rho)}.
\end{align*}
Here $P^{\Gamma}(\rho,\partial_{j}\rho,\partial_{ij}\rho)$ is a polynomial in $\rho$ and its derivatives up to second order with real analytic coefficients. Moreover, $Q^{\Gamma}(\rho)$ is a polynomial in $\rho$ with real analytic coefficients. In particular, we have $Q^{\Gamma}\neq{0}$.
\smallskip\\
In full view of the above computations, within every local chart $K_{\rho}= {\det}[G^{-1}_{\Gamma}(\rho)L^{\Gamma}]$ can be expressed locally as 
\begin{equation}
\label{Gaussian curvature}
K_{\rho}=\beta^{2}(\rho)\frac{P^{\Gamma}(\rho,\partial_{j}\rho,\partial_{ij}\rho)}{ {\det}(G^{\Gamma}(\rho)) Q^{\Gamma}(\rho)}.
\end{equation}
\smallskip\\

As a straightforward conclusion of the above computation, we obtain an explicit expression for $H_{\rho}:=\Psi^{\ast}_{\rho}H_{\Gamma_{\rho}}$:
\begin{align}
\label{mean curvature}
\notag 2 H_{\rho}&=g^{ij}_{\Gamma}l^{\Gamma}_{ij}\\
\notag&=\beta(\rho){g^{ij}_{\Gamma}}[{l_{ij}}-{l_{ik}}{l^{k}_{j}}\rho+{\partial_{ij}\rho}-\Gamma^{k}_{ij}\partial_{k}\rho
+r_k^l(\rho)l^{k}_{i}\partial_{j}\rho\partial_{l}\rho \\
&\hspace*{1em}+r_k^l(\rho)(\partial_{j}l_{i}^{k}+\Gamma^{k}_{jh}l_{i}^{h}-\Gamma^{h}_{ij}l_{h}^{k})\rho\partial_{l}\rho 
+r_k^l(\rho)l^{h}_{j}l^{k}_{h}\rho\partial_{i}\rho\partial_{l}\rho +{l^{k}_{j}}\partial_{i}\rho\partial_{k}\rho]{.}
\end{align}
The reader may also find a different global expression for $H_{\rho}$ in \cite[formula~(32)]{MCH}. We can decompose $H_{\rho}$ into $H_{\rho}=P_1(\rho)\rho+F_1(\rho)$:
\begin{align*}
F_1(\rho)=\frac{\beta(\rho)}{2} {g^{ij}_{\Gamma}}({l_{ij}}-{l_{ik}}{l^{k}_{j}}\rho)=\frac{\beta(\rho)}{2} {\rm{Tr}}[G^{-1}_{\Gamma}(\rho)(L^{\M}-\rho{L^{\M}}L_{\M})],
\end{align*}
where ${\rm{Tr}}(\cdot)$ denotes the trace operator, 
and
\begin{align*}
P_1(\rho)&=\frac{\beta(\rho)}{2}\{{g^{ij}_{\Gamma}}\partial_{ij}
+{g^{ij}_{\Gamma}}({l^{k}_{j}}\partial_{i}\rho-\Gamma^{k}_{ij})\partial_{k}
\\
&\hspace*{1em}+{g^{ij}_{\Gamma}}[r_k^l(\rho)l^{k}_{i}\partial_{j}\rho +
r_k^l(\rho)(\partial_{j}l_{i}^{k}+\Gamma^{k}_{jh}l_{i}^{h}-\Gamma^{h}_{ij}l_{h}^{k})\rho 
+r_k^l(\rho)l^{h}_{j}l^{k}_{h}\rho\partial_{i}\rho ]\partial_{l} \}
\end{align*}
in every local chart. Note that ${\rm{Tr}}[G^{-1}_{\Gamma}(\rho)L^{\M}]$ changes like $H_{\M}$ under transition maps and thus is invariant. Analogously, we can check that $F_1$ is a well-defined global operator. Hence so is $P_1(\rho)$.
\smallskip\\
In addition, it is a well-known fact that $\Psi_{\rho}^{\ast}\Delta_{\Gamma_{\rho}}=\Delta_{\rho}\Psi_{\rho}^{\ast}$, where $\Delta_{\Gamma_{\rho}}$ and $\Delta_{\rho}$ are the Laplace-Beltrami operators on $(\Gamma_{\rho},g_{\Gamma})$ and $({\M},\sigma(\rho))$, respectively. Here $\sigma(\rho):=\Psi_{\rho}^{\ast}g_{\Gamma}$ stands for the pull-back metric of $g_{\Gamma}$ on ${\M}$ by $\Psi_{\rho}$. Then in every local chart, the Laplace-Beltrami operator $\Delta_{\rho}$ can be expressed as
\begin{equation}
\label{Laplacian}
\Delta_{\rho}=\sigma^{jk}(\rho)(\partial_j\partial_k-\gamma^i_{jk}(\rho)\partial_i).
\end{equation} 
Here $\sigma^{jk}(\rho)$ are the components of the induced metric $\sigma^{\ast}(\rho)$ of $\sigma(\rho)$ on the cotangent bundle. Note that $\sigma^{jk}(\rho)$ involves the derivatives of $\rho$ merely up to order one. $\gamma^i_{jk}(\rho)$ are the corresponding Christoffel symbols of $\sigma(\rho)$, which contain the derivatives of $\rho$ up to second order.

There exists a global operator $R(\rho)\in{\mathcal{L}(h^{3+\alpha}({\M}),E_0)}$ such that $R(\cdot)$ is well defined on $\mho$ and:
\begin{align*}
R(\rho)\rho=\frac{1}{2\beta(\rho)}\Delta_{\rho}[\beta(\rho){\rm{Tr}}(G^{-1}_{\Gamma}(\rho)L^{\M})]-\frac{\rho}{2\beta(\rho)}\Delta_{\rho}[\beta(\rho){\rm{Tr}}(G^{-1}_{\Gamma}(\rho){L^{\M}}L_{\M})]{.}
\end{align*}
We set
\begin{align*}
&P(\rho):=\frac{1}{\beta(\rho)}\Delta_{\rho}P_1+R(\rho),\hspace*{13.4em}\rho\in\mho{,}\\
&F(\rho):=-\frac{1}{\beta(\rho)}\Delta_{\rho}F_1(\rho)+R(\rho)\rho-\frac{2}{\beta(\rho)}H_{\rho}(H_{\rho}^{2}-K_{\rho}),\hspace*{1.5em}\rho\in\mho\cap{h^{3+\alpha}({\M})}.
\end{align*}
Note that third order derivatives of $\rho$ do not appear in $F(\rho)$. Hence it is actually well-defined on $\mho$. Based on the above discussion, these two maps enjoy the following smoothness properties:
\begin{center}
$P\in{C}^{\infty}(\mho,\mathcal{L}(E_1,E_0))$\hspace{1em}and\hspace{1em}$F\in{C}^{\infty}(\mho,E_0)$.
\end{center}
\begin{definition}
\label{differential operators}
Let $l\in{\N}_0$. A linear operator $\mathcal{A}:\mathcal{D}({\M})\rightarrow{C}({\M})$ is called a differential operator of order $l$ with continuous coefficients on ${\M}$ if for any $u\in\mathcal{D}({\M})$ it holds that 
\begin{center}
$\psi_{\kappa}^{\ast}(\mathcal{A}u)=\mathcal{A}_{\kappa}(\psi_{\kappa}^{\ast}u)$
\end{center} 
for every local chart $({\Uk},\varphi_{\kappa})$ and some differential operator $\mathcal{A}_{\kappa}=\sum_{|\alpha|\leq{l}}a^{\kappa}_{\alpha}\partial^{\alpha}$ with $a^{\kappa}_{\alpha}\in{C({\Qt})}$ defined on ${\Qt}$, and at least one of the $\mathcal{A}_{\kappa}$'s is of order $l$. In particular, when $l=0$, $\mathcal{A}u=au$ for some $a\in{C({M})}$.
\end{definition}
By the above definition, $P(\rho)$ is a fourth order differential operator with continuous coefficients on ${\M}$ for each $\rho\in\mho$. In every local chart $({\Uk},\varphi_{\kappa})$, the principal part of the local expression of $P(\rho)$  can be written as
\begin{align*}
P^{\pi}_{\kappa}(\rho):=\sigma^{kl}(\rho)g^{ij}_{\Gamma}\partial_{ijkl}.
\end{align*}
Given $\xi\in{T^{\ast}{\M}}$, we estimate the symbol of $P^{\pi}_{\kappa}(\rho)$ as follows.
\begin{align*}
P^{\pi}_{\kappa}(\rho)(\xi)=\sigma^{\ast}(\rho)(\xi,\xi)g_{\Gamma}^{\ast}(\xi,\xi)\geq{c|\xi|^4}
\end{align*}
for some $c>0$, and $g_{\Gamma}^{\ast}$ denotes the induced metric of $g_{\Gamma}$ on the cotangent bundle of ${\M}$. Hence, $P(\rho)$ is a uniformly elliptic fourth order operator acting on functions over ${\M}$ for each $\rho\in\mho$. By \cite[Theorem 4.3]{YS1P}, $P(\rho)\in\mathcal{H}(E_1,E_0)$, namely that $-P(\rho)$ generates an analytic semigroup on $E_0$ with $D(-P(\rho))=E_1$, $\rho\in\mho$. 
\smallskip\\

Now the Willmore flow \eqref{original eq 1.1} can be rewritten as:
\begin{equation}
\label{transformed eq 1.1}
\begin{cases}
\rho_t+P(\rho)\rho=F(\rho),\\
\rho(0)=\rho_0{,}
\end{cases}
\end{equation}
where $\rho_0\in\mho$. A different characterization of the problem can be found in \cite{CMA,SDF,GSWF}.
\smallskip\\

Applying \cite[Theorem 4.1]{MCS}, the existence and regularity result in \cite{GSWF} can be restated as:
\goodbreak
\begin{theorem}
\label{well-posedness}
\cite[Theorem 1.1]{GSWF} Suppose that $\rho_{0}\in{h^{2+\alpha}({\M})}$. Then equation \eqref{transformed eq 1.1} has a unique solution $\rho$ such that
\begin{center}
$\rho\in{C^1_{\frac{1}{2}}(J(\rho_0),E_0)\cap{C_{\frac{1}{2}}(J(\rho_0),E_1)}\cap{C(J(\rho_0),h^{2+\alpha}({\M}))}\cap{C^{\frac{1}{2}-\beta_0}(J(\rho_0),E_{\beta_0})}}$
\end{center}
for any $\beta_0\in[0,\frac{1}{2}]$. Moreover, each hypersurface $\Gamma(t)$ is of class $C^{\infty}$ for $t\in\dot{J}(\rho)$.
\end{theorem}
\bigskip
\vspace{1em}

\section{\bf Parameter-Dependent Diffeomorphisms}
\medskip
The main purpose of the last two sections is to show that the classical solution obtained in Theorem \ref{well-posedness} is in fact real analytic jointly in time and space. To this end, I will construct a family of parameter-dependent diffeomorphisms acting on functions over ${\M}$ first. Because the construction applies to manifolds of arbitrary dimensions, in this section we assume that ${\M}$ is a m-dimensional manifold with the properties imposed in Section 1.

For a given point $p\in{\M}$, we choose a normalized atlas $({\Uk},\varphi_{\kappa})_{{\kappa}\in {\Lambda}}$ for ${\M}$ such that $\varphi_{1}(p)=0\in{\R}^m$. Choose several open subsets $B_{i}$ in ${\Q}$, the open unit ball centered at the origin in $\R^m$, in such a manner that:
\begin{itemize}
\item  $B_{i}:=\mathbb{B}^{m}(0,i\varepsilon_{0})$,\hspace{.5em} for $i=1,2,3$ and some $\varepsilon_{0}>0$.
\item  ${B_{3}}\subset\subset{B_{4}}\subset\subset{\Q}$.
\end{itemize}
\smallskip
Next, I further pick two cut-off functions on ${\Q}$:
\begin{itemize}
\item  $\chi\in{\mathcal{D}(B_{2},[0,1])}$ such that $\chi|_{\overline{B}_{1}}\equiv{1}$. We write $\chi_{\kappa}=\varphi^{\ast}_{\kappa}\chi$.
\item  $\zeta\in\mathcal{D}(B_{4},[0,1])$ such that $\tilde{\chi}|_{\overline{B}_{3}}\equiv{1}$. We write $\zeta_{\kappa}=\varphi^{\ast}_{\kappa}\zeta$.
\end{itemize}

We define a rescaled translation on ${\Q}$ for any ${\mu}\in{\B}\subset{\R}^m$ with $r$ sufficiently small:
\begin{center}
$\theta_{\mu}(x):=x+\chi{(x)}\mu$,\hspace{1em} $x\in{\Q}$.
\end{center}
\smallskip
This localization technique in Euclidean spaces was first introduced in \cite{ARP} by J.~Escher, J.~Pr\"uss and G.~Simonett to establish regularity for solutions to parabolic and elliptic equations.
\smallskip\\
Given a function $v\in{L_{1,loc}({\Q})}$, its pull-back and push-forward induced by $\theta_{\mu}$ are defined as:
\begin{center}
${\theta^{\ast}_{\mu}}v:=v\circ{\theta_{\mu}}$\hspace{1em} and \hspace{1em}$\theta^{\mu}_{\ast}v:=v\circ{\theta_{\mu}^{-1}}$.
\end{center}

The diffeomorphism $\theta_{\mu}$ induces a transformation $\Theta_{\mu}$ on ${\M}$ by:
\begin{align*}
\Theta_{\mu}(q)=
\begin{cases}
\psi_1(\theta_{\mu}(\varphi_1(q))) \hspace{1em}&q\in{\U}_1,\\
q &q\notin{\U}_1.
\end{cases}
\end{align*}
\smallskip\\
It can be shown that $\Theta_{\mu}\in{\rm{Diff}^{\hspace{.2em}\infty}}({\M})$ for $\mu\in{\B}$ with sufficiently small $r>0$. See \cite{YS1P} for details.
\smallskip\\
For any $u\in{L_{1,loc}({\M})}$, we can define its pull-back and push-forward induced by $\Theta_{\mu}$ analogously as:
\begin{center}
${\ttm}u:=u\circ{\Theta_{\mu}}$\hspace{1em} and \hspace{1em}$\Theta^{\mu}_{\ast}u:=u\circ{\Theta_{\mu}^{-1}}$.
\end{center}
We may find an explicit global expression for the transformation ${\ttm}$ on ${\M}$,  
\begin{center}
${\ttm}u={\varphi^{\ast}_{1}}{\theta^{\ast}_{\mu}}{\psi_{1}^{\ast}}({\Xt}u)+(1-{\Xt})u$.
\end{center}
Here and in the following it is understood that a partially defined and compactly supported function is automatically extended over the whole base manifold by identifying it to be zero outside its original domain. 
\smallskip\\
Likewise, we can express ${\Theta}^{\mu}_{\ast}$ as
\begin{center}
${\Theta}^{\mu}_{\ast}={\varphi^{\ast}_{1}}{\theta^{\mu}_{\ast}}{\psi_{1}^{\ast}}({\Xt}u)+(1-{\Xt})u$.
\end{center}
\smallskip
Let $I=[0,T]$, $T>0$. Assuming that $J\subset(0,T)$ is an open interval and $t_{0}\in{J}$ is a fixed point, we choose $\varepsilon_{0}$ to be so small that $\mathbb{B}(t_{0},3\varepsilon_{0})\subset{J}$. Next we pick another auxiliary function 
\begin{center}
$\xi\in\mathcal{D}(\mathbb{B}(t_{0},2\varepsilon_{0}),[0,1])$\hspace{1em} with\hspace{.5em} $\xi|_{\mathbb{B}(t_{0},\varepsilon_{0})}\equiv{1}$. 
\end{center}
The above construction now engenders a parameter-dependent transformation in terms of the time variable:
\begin{center}
$\varrho_{\lambda}(t):=t+\xi(t)\lambda$,\hspace{.5em} for any $t\in{I}$ and $\lambda\in{\R}$.
\end{center}
\smallskip
Now we are in a situation to define a family of parameter-dependent transformations on $I\times{\M}$. Given a function $u:I\times{\M}\rightarrow{\R}$, we set
\begin{center}
${{u}_{\lambda,\mu}}(t,\cdot):={\ttl}u(t,\cdot):={\tu}(t){\rh}u(t,\cdot)$,
\end{center}
where ${\tu}(t)={\Theta}^{\ast}_{\xi(t)\mu}$ and $(\lambda,\mu)\in{\B}$.
\smallskip\\
It is important to note that ${{u}_{\lambda,\mu}}(0,\cdot)=u(0,\cdot)$ for any $(\lambda,\mu)\in{\B}$ and any function $u$.

The importance of this family of parameter-dependent diffeomorphisms lies in the following theorems. Their proofs as well as additional properties of this technique can be found in \cite{YS1P}.
\begin{theorem}
\label{inverse}
Let $k\in{\N}_{0}\cup\{\infty,\omega\}$. Suppose that $u\in{C(I\times{\M})}$. Then we have that $u\in{C^{k}(\dot{I}\times{\M})}$ iff for any $(t_0,p)\in\dot{I}\times{\M}$, there exists $r=r(t_0,p)>0$ and a corresponding family of parameter-dependent diffeomorphisms ${\ttl}$ such that 
\begin{center}
$[(\lambda,\mu)\mapsto{\ttl}u]\in{C^{k}({\B},C(I\times{\M}))}$.
\end{center}
Here $\omega$ is the symbol for real analyticity. 
\end{theorem}
\bigskip
\goodbreak

\begin{prop}
\label{time derivatives}
Suppose that $u\in{\ef}(I)$. Then $u_{\lambda,\mu} \in{\ef}(I)$, and 
\begin{align*}
\partial_t[u_{\lambda,\mu}]=(1+\xi^{\prime}\lambda){\ttl}u_t+B_{\lambda,\mu}(u_{\lambda,\mu}),
\end{align*}
where 
\begin{center}
$[(\lambda,\mu)\mapsto{B}_{\lambda,\mu}]\in{C}^{\omega}({\B},C(I,\mathcal{L}(E_1,E_0)))$. 
\end{center}
Furthermore, $B_{\lambda,0}=0$.
\end{prop}
\bigskip
\goodbreak
\begin{prop}
\label{regularity of differential operators}
Let $s\in[0,h]$ and $l\in\N_0$. Suppose that $\mathcal{A}$ is a differential operator of order $l$ with continuous coefficients on ${\M}$ satisfying $a^{\kappa}_{\alpha}\in{BC^h({\Q})}$ and $a^{1}_{\alpha}\in{BC^h({\Q})}\cap{C}^{\omega}(\sf{O})$ for some open subset $\sf{O}$ such that $B_3\subset\subset{\sf{O}}\subset\subset{\Q}$. Then 
\begin{center}
$[\mu\mapsto{T}_{\mu}\mathcal{A}T_{\mu}^{-1}]\in{C}^{\omega}({\B},C(I,\mathcal{L}(h^{s+l}({\M}),h^{s}({\M}))))$.
\end{center} 
\end{prop}
\bigskip
\goodbreak
\begin{prop}
\label{regularity of transformed functions}
Let $s\geq 0$. Suppose that $u\in {C^{\omega}(\psi_1(\U))\cap \mathit{h^{s}}({\M})}$,
where $\U$ is defined in Proposition \ref{regularity of differential operators}. Then
\begin{center}
$[\mu\mapsto{\tu}u]\in{C^{\omega}({\B},C(I,{h}^{s}({\M})))}$.
\end{center}
\end{prop}
\bigskip
\vspace{1em}

\section{\bf Real Analyticity}
\medskip
By setting $G(\rho):=P(\rho)\rho-F(\rho)$, we may rewrite equation \eqref{transformed eq 1.1} as 
\begin{equation}
\label{final eq 1.1}
\begin{cases}
\rho_t+G(\rho)=0{,}\\
\rho(0)=\rho_0{.}
\end{cases}
\end{equation}
\begin{theorem}
\label{main theorem 2}
Let $0<\alpha<1$. Suppose that $\rho_{0}\in{h^{2+\alpha}({\M})}$. Then equation \eqref{final eq 1.1} has a unique local solution $\rho$ in the interval of maximal existence $J(\rho_{0})$ such that 
\begin{center}
$\rho\in{C^{\omega}(\dot{J}(\rho_{0})\times{\M})}$.
\end{center}
\end{theorem}
\begin{proof}
I will indicate herein all the key steps of the proof. More details can be found in \cite{YS1P}.

For any $(t_0,p)\in\dot{J}(\rho_0)\times{\M}$ and sufficiently small $r>0$, a family of parameter-dependent diffeomorphisms ${\ttl}$ can be defined for $(\lambda,\mu)\in{\B}$. Henceforth, we always use the notation $\rho$ exclusively for the solution to \eqref{transformed eq 1.1} and hence to \eqref{final eq 1.1}. Set $u:={\rho}_{\lambda,\mu}$. Then as a consequence of Proposition \ref{time derivatives}, $u$ satisfies the equation
\begin{align*}
u_t&=\partial_t[{\rho}_{\lambda,\mu}]=(1+\xi^{\prime}\lambda){\ttl}\rho_t+{B}_{\lambda,\mu}(u)\\
&=-(1+\xi^{\prime}\lambda){\ttl}G(\rho)+{B}_{\lambda,\mu}(u)\\
&=-(1+\xi^{\prime}\lambda){\tu}G({\rh}\rho)+{B}_{\lambda,\mu}(u)\\
&=-(1+\xi^{\prime}\lambda){\tu}G({T}^{-1}_{\mu}u)+{B}_{\lambda,\mu}(u):=-H_{\lambda,\mu}(u).
\end{align*}
Pick $I:[\varepsilon,T]\subset\subset{J(\rho_{0})}$ such that $t_{0}\in\dot{I}$ and $\mathbb{B}(t_0,3\varepsilon_0)\subset\subset\dot{I}$. Then we define ${\ez}(I)$ and ${\ef}(I)$ as in Section~1 by moving the initial point from $0$ to $\varepsilon$. Set 
\begin{align*}
\efa(I):=\{v\in{\ef}(I):\|v\|_{\infty}<a\},
\end{align*}
where $\|v\|_{\infty}:=\sup_{(t,q)\in I\times{\M}} |v(t,q)|$.
\smallskip\\
For $\mathcal{A}\in\mathcal{H}(E_1,E_0)$, we say that $({\ez}(I),{\ef}(I))$ is a pair of maximal regularity of $\mathcal{A}$, if
\begin{center}
$(\frac{d}{dt}+\mathcal{A},\gamma_{\varepsilon})\in{\rm{Isom}}({\ef}(I),{\ez}(I)\times{E_1})$,
\end{center}
where $\gamma_{\varepsilon}$ is the evaluation map at $\varepsilon$, i.e., $\gamma_{\varepsilon}(u)=u(\varepsilon)$. Next we define 
\begin{center}
$\Phi:{\efa}(I)\times{\B}\rightarrow{\ez}(I)\times{E_1}$\hspace{.5em} as\hspace{.5em} $\Phi(v,(\lambda,\mu))\mapsto\dbinom{v_t+H_{\lambda,\mu}(v)}{\gamma_{\varepsilon}(v)-\rho(\varepsilon)}$.
\end{center}
Note that $\Phi({\rho}_{\lambda,\mu},(\lambda,\mu))=\dbinom{0}{0}$ for any $(\lambda,\mu)\in{\B}$.

\noindent(i) My first goal is to prove that
$\Phi\in{C^{\omega}({\efa}(I)\times{\B},{\ez}(I)\times{E_1})}$.

By Proposition \ref{time derivatives}, ${B}_{\lambda,\mu}\in{C^{\omega}({\B},C(I,\mathcal{L}(E_1,E_0)))}$. We define a bilinear and continuous map:
\begin{center}
$f:C(I,\mathcal{L}(E_1,E_0))\times{{\ef}(I)}\rightarrow{\ez}(I)$,\hspace{.5em} $(T(t),u(t))\mapsto{T(t)(u(t))}$.
\end{center}
Hence $[(v,(\lambda,\mu))\mapsto{f({B}_{\lambda,\mu},v)}={B}_{\lambda,\mu}(v)]\in{C^{\omega}({\efa}(I)\times{\B},{\ez}(I))}$.

On the other hand, let $\pi=\sum_{\eta\in\mathcal{C}(1)}\pi_{\eta}^2$, where 
\begin{align*}
\mathcal{C}(1):=\{\eta\in\Lambda: {\rm{supp}}(\pi_{\eta})\cap{{\rm{supp}}(\pi_1)\not=\emptyset}\}.
\end{align*}
We decompose $G$ into
\begin{center}
$G=\pi{G}+\sum_{\eta\notin\mathcal{C}(1)}\pi_{\eta}^2G$.
\end{center}
According to our construction of ${\ttm}$ and of the localization system, we may assume that $\pi|_{\U}\equiv{1}$, where $\sf{O}$ is defined in Proposition \ref{regularity of differential operators} with $m=2$. See \cite[Lemma~3.2]{FSM} for details. 
\smallskip\\
Taking into account \eqref{Gaussian curvature}, \eqref{mean curvature} and \eqref{Laplacian}, in every local chart $({\Uk},\varphi_{\kappa})$ and for any $v\in{\efa}(I)$, $G(v)$ can be expressed as 
\begin{align*}
\frac{\beta^{2h}(v)P^{G}(v,\cdots,\partial_{ijkl}v)}{{ {\det}(G^{\Gamma}(v))^{s_1}} {\det}([\sigma(v)])^{s_2}Q^{G}(v)},
\end{align*}
where $h,s_1,s_2\in{\N}$. $[\sigma(v)]$ is the matrix representation of the metric $\sigma(v)$. Here $\sigma(v)$ is defined in a similar manner to $\sigma(\rho)$ with $\rho$ replaced by $v$. Analogously, $G^{\Gamma}(v)$ is defined in a similar way to $G^{\Gamma}(\rho)$. Meanwhile, $P^{G}$ is a polynomial in $v$ and its derivatives up to fourth order with real analytic coefficients, and $Q^{G}$ is a polynomial in $v$ with real analytic coefficients.
In particular, $ {\det}([\sigma(v)])$ only involves first order derivatives of $v$.
\smallskip\\
Therefore, $\pi{G}(v)$ can be decomposed globally into
\begin{align*}
\frac{\mathcal{P}_0+\mathcal{P}^1_1{v}\cdots\mathcal{P}^1_{k_1}{v}+\cdots+\mathcal{P}^r_1{v}\cdots\mathcal{P}^r_{k_r}{v}}{\mathcal{Q}_0+\mathcal{Q}^1_1{v}\cdots\mathcal{Q}^1_{l_1}{v}+\cdots+\mathcal{Q}^s_1{v}\cdots\mathcal{Q}^s_{l_s}{v}}, 
\end{align*}
where $\mathcal{P}_0$, $\mathcal{Q}_0\in C^{\infty}({\M})\cap C^{\omega}(\psi_1(\U))$. The $\mathcal{P}^i_{j_i}$'s are linear differential operators with continuous coefficients on ${\M}$ up to fourth order, and the $\mathcal{Q}^i_{j_i}$'s are linear differential operators of order at most one with continuous coefficients on ${\M}$. Their coefficients in every local chart satisfy that $a^{\kappa}_{\alpha}\in{BC^{\infty}({\Qt})}$ and $a^{1}_{\alpha}\in{BC^{\infty}({\Qt})}\cap{C^{\omega}({\U})}$. By Proposition \ref{regularity of transformed functions}, we deduce that 
\begin{align*}
[\mu\mapsto(\tu\mathcal{P}_0, \tu\mathcal{Q}_0)]\in C^{\omega}({\B},C(I,E_1)\times C(I,E_1)).
\end{align*}
Analogously, it follows from Proposition \ref{regularity of differential operators} that 
\begin{align*}
[\mu\mapsto{\tu}\mathcal{P}^i_{j_i}T^{-1}_{\mu}]\in{C^{\omega}({\B},C(I,\mathcal{L}(E_1,E_0)))}
\end{align*} 
and
\begin{align*}
[\mu\mapsto{\tu}\mathcal{Q}^i_{j_i}T^{-1}_{\mu}]\in{C^{\omega}({\B},C(I,\mathcal{L}(E_1,h^{3+\alpha}({\M}))))}.
\end{align*}
Combining the above discussion with point-wise multiplication theorems on Riemannian manifolds, we infer that 
\begin{align*}
[(v,\mu)\mapsto{\tu}(\pi{G})T^{-1}_{\mu}v]\in{C^{\omega}({\efa}(I)\times{\B},{\ez}(I))}.
\end{align*}
Applying these arguments repeatedly to the other terms $\pi_{\eta}^{2}G$, we conclude that
\begin{align*}
\Phi\in{C^{\omega}({\efa}(I)\times{\B},{\ez}(I)\times{E_1})}.
\end{align*}

\noindent(ii) Next we look at the Fr\'echet derivative of $\Phi$ in the first component:
\begin{align*}
D_1\Phi(v,(\lambda,\mu))w=\dbinom{w_t+(1+\xi^{\prime}\lambda){\tu}DG({T}^{-1}_{\mu}v){T}^{-1}_{\mu}w-{B}_{\lambda,\mu}(w)}{\gamma_{\varepsilon}w}.
\end{align*}
Thus
\begin{align*}
D_1\Phi(\rho,(0,0))w=\dbinom{w_t+DG(\rho)w}{\gamma_{\varepsilon}w}.
\end{align*}
Observe that $DG(\rho)$ is a fourth order linear differential operator whose coefficients satisfy $a^{\kappa}_{\alpha}\in{E_0}$. The principal part of $DG(\rho)$ in every local chart coincides with that of $P(\rho)$, that is, $P^{\pi}_{\kappa}(\rho)$. By the discussion in Section~2, we know that $DG(\rho(t,\cdot))$ is a uniformly elliptic operator for every fixed $t \geq 0$. As a consequence of \cite[Theorem~4.5, Proposition~4.7]{YS1P}, it follows that $({\ez}(I),{\ef}(I))$ is a pair of maximal regularity for $DG(\rho(t,\cdot))$.
\smallskip\\
We set $\mathcal{A}(t)=DG(\rho(t,\cdot))$. It follows that
\begin{align*}
(\frac{d}{dt}+\mathcal{A}(s),\gamma_{\varepsilon})\in{\rm{Isom}}({\ef}({I}),{\ez}({I})\times{E_1}),\hspace*{.5em}\text{ for every } s\in{{I}}.
\end{align*}
By \cite[Lemma 2.8(a)]{MCS}, we have
\begin{align*}
(\frac{d}{dt}+\mathcal{A}(\cdot),\gamma_{\varepsilon})\in{\rm{Isom}}({\ef}(\mathit{I}),{\ez}(\mathit{I})\times{E_1}).
\end{align*}
\smallskip\\
Now we are in a position to apply the Implicit Function Theorem. It follows right away that there exists an open neighborhood, say $\mathbb{B}(0,r_0)\subset{\B}$, such that
\begin{align*}
[(\lambda,\mu)\mapsto{\rho}_{\lambda,\mu}]\in{C^{\omega}(\mathbb{B}(0,r_0),{\ef}(I))}.
\end{align*}
As a consequence of Theorem \ref{inverse}, we deduce that $\rho\in{C^{\omega}(\dot{J}(\rho_0)\times{\M})}$. This completes the proof.
\end{proof}

\begin{proof}
[Proof of Theorem \ref{main theorem}]
For each $(t_0,q)\in\mathcal{M}=\bigcup_{t\in\dot{J}(\rho_0)}(\{t\}\times\Gamma(t))$, there exists a $p\in{\M}$ such that $\Psi_{\rho}(t_0,p)=q$. Here $\Gamma(t)={\rm{im}}(\Psi_{\rho}(t,\cdot))$. Theorem \ref{main theorem 2} states that there exists a local patch $(\Uk,\varphi_{\kappa})$ such that $p\in\Uk$ and $\rho\circ\psi_{\kappa}$ is real analytic in $\dot{J}(\rho_0)\times{\Qt}$. 
\smallskip
Therefore, we conclude that
\begin{align*}
[(t,x)\mapsto (t,\psi_{\kappa}(x)+\rho(t,\psi_{\kappa}(x))\nu_{\M}(\psi_{\kappa}(x))]\in{C^{\omega}(\dot{J}(\rho_0)\times{\Qt},\mathcal{M})}.
\end{align*}
This proves the assertion of Theorem \ref{main theorem}.
\end{proof}
\bigskip

\end{document}